\documentclass[11pt, reqno,a4paper]{article}
\usepackage{amsfonts,anysize,hyperref,amsmath,indentfirst,geometry,amsthm}
\usepackage{bm}
\usepackage{amssymb}
\usepackage{mathrsfs}
\usepackage{amssymb}
\marginsize{25mm}{25mm}{15mm}{15mm}

\allowdisplaybreaks
\numberwithin{equation}{section}
\usepackage{graphicx}
\usepackage{amsmath}
\usepackage{amssymb}
\usepackage{amsfonts,amssymb,amsbsy,amsmath}
\usepackage{eqnarray}
\usepackage{mathptmx}
\newtheorem{theorem}{Theorem}[section]
\newtheorem{lemma}{Lemma}[section]

\newtheorem{remark}{Remark}[section]

\allowdisplaybreaks
\begin{document}
\title{\Large \bfseries New rigorous perturbation bounds for the generalized Cholesky factorization\thanks{The work is supported by the National Natural Science Foundation of China (Grant Nos. 11201507; 11171361).}}
\author{{Hanyu Li\thanks{Corresponding author. Email addresses: lihy.hy@gmail.com or hyli@cqu.edu.cn; yangyanfei2008@yeah.net}, Yanfei Yang}}
\date{\small {College of Mathematics and Statistics, Chongqing University,
Chongqing, 401331, P. R. China}} \maketitle
{\raggedleft\bfseries{\normalsize Abstract}}\\

Some new rigorous perturbation bounds for the generalized Cholesky factorization with normwise or componentwise perturbations in the given matrix are obtained, where the componentwise perturbation has the form of backward rounding error for the generalized Cholesky factorization algorithm. These bounds can be much tighter than some existing ones while the conditions for them to hold are simple and moderate.


{\raggedleft \em AMS classification:}\ 15A23; 15A45\\

{\raggedleft \em Keywords:}  Generalized Cholesky factorization; Rigorous perturbation bound; Normwise perturbation; Componentwise perturbation

\section{Introduction}
Let $\mathbb{R}^{m \times n} $ be the set of $m \times n$ real matrices and
$\mathbb{R}^{m\times n}_r $ be the subset of $\mathbb{R}^{m \times n} $ consisting of
matrices with rank $r$. Let $I_r$ be the identity matrix of order $r$ and $A^T$ be the transpose of the matrix $A$.

Consider the following block matrix
\begin{eqnarray}\label{eq:1}
K = \left[ {\begin{array}{*{20}c}
   A & {B^T }  \\
   B & { - C}  \\
\end{array}} \right]\in\mathbb{R}^{(m+n) \times (m+n)},
\end{eqnarray}
where $A \in \mathbb{R}_m^{m \times m}$ is symmetric positive definite, $B \in \mathbb{R}^{n \times m}_n$, and $C \in \mathbb{R}^{n \times n}$ is symmetric positive semi-definite. For this matrix, there always exists the following factorization
\begin{eqnarray}
K = LJ_{m + n} L^T,
\end{eqnarray}
where
\begin{eqnarray*}
L = \left[ {\begin{array}{*{20}c}
   {L_{11} } & 0  \\
   {L_{21} } & {L_{22} }  \\
\end{array}} \right],\quad
J_{m + n}  = \left[ {\begin{array}{*{20}c}
   {I_m } & 0  \\
   0 & { - I_n }  \\
\end{array}} \right],
\end{eqnarray*}
$L_{11}  \in \mathbb{R}^{m \times m}_m$ and $L_{22}  \in \mathbb{R}^{n \times n}_n$ are lower triangular, and $L_{21}  \in \mathbb{R}_n^{n \times m}$. The factorization (1.2) is called the generalized Cholesky factorization and $L$ is referred to as the generalized Cholesky factor [1]. If the diagonal elements of the lower triangular matrices $L_{11}$ and $L_{22}$ are positive, the factorization is unique.

For the generalized Cholesky factorization, some scholars considered its applications, algorithms, algorithms' numerical stability, and perturbation analysis [1--6]. Several first-order perturbation bounds were presented [2, 4--6]. Since, in some cases, it is unclear whether the first-order bound is a good approximate bound as it ignores the higher-order terms, we have to be careful to use them in practice. Oppositely, the rigorous perturbation bounds can be used safely for all cases. So it is important to derive the rigorous bounds for the generalized Cholesky factorization. At present, some rigorous bounds have been given for this factorization [3, 4, 6]. However, these bounds are either quite loose or derived under more restrictive conditions or expensive to compute. The rigorous bounds derived in this paper, using the combination of the classic and refined matrix equation approaches [7], overcome these disadvantages to some extent. They can be much tighter than some existing bounds while the conditions for them to hold are simple and moderate.

The rest of this paper is organized as follows. Section 2 presents some notation and basics. The rigorous perturbation bounds with normwise or componentwise perturbations are given in Sections 3 and 4, respectively. Finally, the concluding remarks of the whole paper is provided.

\section{ Notation and basics}

Given a matrix $A \in \mathbb{R}_r^{m \times n}$, $\left\| A \right\|_2$ and $\left\| A \right\|_F$ stand for its spectral norm and Frobenius norm, respectively. From [8, pp. 80], we have
\begin{eqnarray}
\left\| {XYZ} \right\|_2  \le \left\| X \right\|_2 \left\| Y \right\|_2 \left\| Z \right\|_2 ,\quad \ \left\| {XYZ} \right\|_F  \le \left\| X \right\|_2 \left\| Y \right\|_F \left\| Z \right\|_2,
\end{eqnarray}
whenever the matrix product $XYZ$ is defined. If $A$ is nonsingular, we denote its standard condition number by $\kappa_2(A) = \left\| {A^{ - 1} } \right\|_2 \left\| A \right\|_2$ and Bauer--Skeel condition number by ${\rm{cond}}_F(A) = \left\| {| {{A^{ - 1}}} |} {| A |} \right\|_F$ [8, pp. 128]. Here, for any matrix $X=(x_{ij})$, $|X|$ is defined by $|X|=(|x_{ij}|)$.

For any matrix $A = \left( {a_{ij} } \right) \in \mathbb{R}^{n \times n}$, define
\begin{eqnarray*}
{\rm{up}}\left( A \right) = \left[ {\begin{array}{*{20}c}
   {\frac{1}{2}a_{11} } & {a_{12} } &  \cdots  & {a_{1n} }  \\
   0 & {\frac{1}{2}a_{22} } &  \cdots  & {a_{2n} }  \\
    \vdots  &  \vdots  &  \ddots  &  \vdots   \\
   0 & 0 &  \cdots  & {\frac{1}{2}a_{nn} }  \\
\end{array}} \right].
\end{eqnarray*}
Obviously,
\begin{eqnarray}
\left\| {{\rm{up}}\left( A \right)} \right\|_F  \le \left\| A \right\|_F.
\end{eqnarray}
If $A^T  = A$, from [9], we have
\begin{eqnarray}
\left\| {{\rm{up}}\left( A \right)} \right\|_F  \le \frac{1}{{\sqrt 2 }}\left\| A \right\|_F.
\end{eqnarray}
Moreover, let $\mathbb{D}_{n}  \in \mathbb{R}^{n \times n}$ be the set of $n \times n$ positive definite diagonal matrices. Then, for any $D_{n}  \in \mathbb{D}_{n}$,
\begin{eqnarray}
{\mathop{\rm up}\nolimits} \left( {AD_{n} } \right) = {\mathop{\rm up}\nolimits} \left( A \right)D_{n}.
\end{eqnarray}

The following lemma is needed later in this paper, which is taken from [7].

\begin{lemma}
Let $a,b > 0$, and $c(\cdot)$ be a continuous function of a parameter $t \in \left[ {0,1} \right]$ such that $
b^2  - 4ac\left( t \right) > 0$ holds for all $t$. Suppose that a continuous function $x\left( t \right)$ satisfies the quadratic inequality $ax^2\left( t \right)  - bx\left( t \right) + c\left( t \right) \ge 0$. If $c\left( 0 \right) = x\left( 0 \right) = 0$, then
\begin{eqnarray*}
x\left( 1 \right) \le \frac{1}{{2a}}\left( {b - \sqrt {b^2  - 4ac\left( 1 \right)} } \right).
\end{eqnarray*}
\end{lemma}

\section{Rigorous perturbation bounds with normwise perturbation}

The main theorem, similar to Theorem 3.1 of [7], which is concerned with the regular Cholesky factorization, is presented as follows.

\begin{theorem}
Let $K \in \mathbb{R}^{\left( {m + n} \right) \times \left( {m + n} \right)}$ be expressed as in {\rm(1.1)} and factorized as in {\rm(1.2)}.
Let $\Delta K \in \mathbb{R}^{\left( {m + n} \right) \times \left( {m + n} \right)}$ be symmetric. If
\begin{eqnarray}
\left\| {L^{ - 1} } \right\|_2^2 \left\| {\Delta K} \right\|_F  < \frac{1}{2},
\end{eqnarray}
then $K + \Delta K$ has the following generalized Cholesky factorization
\begin{eqnarray}
K + \Delta K = \left( {L + \Delta L} \right)J_{m + n} \left( {L + \Delta L} \right)^T.
\end{eqnarray}
Moreover,
\begin{eqnarray}
&&\left\| {\Delta L} \right\|_F\le \frac{\sqrt 2\left\| {L^{ - 1} } \right\|_2\left[ {\mathop {\inf }\limits_{D_{m + n} \in \mathbb{D}_{m + n} } \kappa_2 (LD_{m + n}^{ - 1} )} \right]}{{\sqrt 2  - 1 + \sqrt {1 - 2\left\| {L^{ - 1} } \right\|_2^2 \left\| {\Delta K} \right\|_F } }}{{\left\| {\Delta K} \right\|_F }}\\
&&\quad\quad\quad\le ( {2 + \sqrt 2 })\left\| {L^{ - 1} } \right\|_2\left[ {\mathop {\inf }\limits_{D_{m + n}  \in \mathbb{D}_{m + n} } \kappa_2(LD_{m + n}^{ - 1} )} \right]{{\left\| {\Delta K} \right\|_F }}.
\end{eqnarray}
\end{theorem}
\begin{proof}
Using (2.1) and noting the condition (3.1), we have that for any $t \in \left[ {0,1} \right]$,
\begin{eqnarray*}
\rho({L^{ - 1} t(\Delta K)L^{ - T} })\leq\left\| {L^{ - 1} t(\Delta K)L^{ - T} } \right\|_F \leq \left\| {L^{ - 1} } \right\|_2^2 \left\| {\Delta K} \right\|_F  < \frac{1}{2}.
\end{eqnarray*}
Here, for any square matrix $X$, $\rho(X)$ denotes its spectral radius. Thus, from Theorem 2.2  in [2] and its proof or Theorem 2.1 in [6] and its proof, it follows that the matrix $K + t(\Delta K)$ has the following generalized Cholesky factorization
\begin{eqnarray}
K + t(\Delta K) = L(t)J_{m + n}L^T(t)=\left( {L + \left( \Delta L(t) \right)} \right)J_{m + n} \left( {L + \left( \Delta L(t) \right)} \right)^T,
\end{eqnarray}
where $L(t)$ is lower triangular having the same structure as that of $L$ in (1.2) and $\Delta L\left( t \right)=L(t)-L$ with $ \Delta L\left(0 \right)=0$. Setting $\Delta L\left( 1 \right) = \Delta L$ in (3.5) gives (3.2).

In the following, we consider (3.3) and (3.4). Observing (1.2), it follows from (3.5) that
\begin{eqnarray*}
t(\Delta K) = LJ_{m + n} (\Delta L (t) )^T+ (\Delta L( t))J_{m + n} L^T  + (\Delta L( t ))J_{m + n}( \Delta L(t))^T.
\end{eqnarray*}
Left-multiplying the above equation by  $L^{ - 1}$ and right-multiplying it by  $L^{ - T}$ gives
\begin{align}
&J_{m + n} (\Delta L( t))^T L^{ - T}  + L^{ - 1} \left(\Delta L(t)  \right)J_{m + n}= tL^{ - 1} (\Delta K)L^{ - T}- L^{ - 1} \left( \Delta L(t) \right)J_{m + n} (\Delta L ( t))^TL^{ - T}.
\end{align}
Since $J_{m + n} \left(\Delta L ( t )\right)^TL^{ - T}$ is upper triangular, using the symbol "up," from (3.6), we have
\begin{eqnarray}
J_{m + n} (\Delta L(t))^TL^{ - T}  = {\rm up}\left[ {tL^{ - 1} (\Delta K)L^{ - T}  - L^{ - 1}\left( \Delta L( t) \right)J_{m + n} (\Delta L( t))^TL^{ - T} } \right].
\end{eqnarray}
Taking the Frobenius norm on (3.7) and considering (2.3) and (2.1) leads to
\begin{eqnarray*}
&&\left\| {L^{ - 1}\left(\Delta L( t) \right) } \right\|_F = \left\|J_{m + n}  {(\Delta L (t))^TL^{ - T} } \right\|_F  \le \frac{1}{{\sqrt 2 }}\left\| {tL^{ - 1}( \Delta K)L^{ - T}  -  L^{ - 1} \left( \Delta L(t) \right)J_{m + n} (\Delta L ( t))^TL^{ - T} } \right\|_F\\
&&\quad\quad\quad\quad\quad\quad\le \frac{1}{{\sqrt 2 }}\left( {\left\| {{L^{ - 1}}} \right\|_2^2{{\left\| {\Delta K} \right\|}_F}t + \left\| {{L^{ - 1}} \left(\Delta L( t) \right)} \right\|_F^2} \right).
\end{eqnarray*}
Let $x(t)=\left\| {{L^{ - 1}}\left( \Delta L(t) \right)} \right\|_F$ and $c(t)=\left\| {{L^{ - 1}}} \right\|_2^2{\left\| {\Delta K} \right\|_F}t$. It is easy to find that both $x(t)$ and $c(t)$ are continuous with respect to $t$. Moreover,
\begin{eqnarray*}
x^2(t)-\sqrt{2}x(t)+c(t) \ge 0.
\end{eqnarray*}
From (3.1), it is seen that for any $t \in \left[ {0,1} \right]$, $ 2 - 4c(t)> 0$.
Meanwhile, $x(0)=0$ and $c(0) = 0$. These facts mean that all the conditions of Lemma 2.1 hold. Thus, by Lemma 2.1, we have
\begin{align}
\left\| {L^{ - 1} (\Delta L)} \right\|_F =x(1) \le  \frac{1}{{2 }}\left( {\sqrt{2} - \sqrt {2 - 4c(1) } } \right)=\frac{1}{{\sqrt 2 }}\left( {1 - \sqrt {1 - 2\left\| {L^{ - 1} } \right\|_2^2 \left\| {\Delta K} \right\|_F } } \right).
\end{align}

Now we introduce a scaling matrix ${D_{m + n}} \in {\mathbb{D}_{m + n}}$ into the expression (3.7) with $t=1$, which can be used to improve the bounds. Right-multiplying (3.7) with $t=1$ by $D_{m + n}$ and using (2.4) yields
\begin{align}
J_{m + n}(\Delta L)^T L^{ - T} D_{m + n}  = {\rm up}\left( { L^{ - 1} (\Delta K)L^{ - T} D_{m + n}  - L^{ - 1} (\Delta L)J_{m + n} (\Delta L)^T L^{ - T} D_{m + n} } \right).
\end{align}
Taking the Frobenius norm on (3.9) and noting (2.2) and (2.1), we get
\begin{eqnarray*}
{\left\| {(\Delta L)^T{L^{ - T}}{D_{m + n}}} \right\|_F} \le {\left\| {{L^{ - 1}}} \right\|_2}{\left\| {{L^{ - T}}{D_{m + n}}} \right\|_2}{\left\| {\Delta K} \right\|_F} + {\left\| {{L^{ - 1}}\Delta L} \right\|_F}{\left\| {(\Delta L)^T{L^{ - T}}{D_{m + n}}} \right\|_F},
\end{eqnarray*}
which combined with (3.8) gives
\begin{eqnarray}
{\left\| {(\Delta {L)^T}{L^{ - T}}{D_{m + n}}} \right\|_F} \le \frac{{{{\left\| {{L^{ - 1}}} \right\|}_2}{{\left\| {{L^{ - T}}{D_{m + n}}} \right\|}_2}{{\left\| {\Delta K} \right\|}_F}}}{1-\left\| {L^{ - 1}}\Delta L \right\|_F}\le \frac{{\sqrt 2 {{\left\| {{L^{ - 1}}} \right\|}_2}{{\left\| {{L^{ - T}}{D_{m + n}}} \right\|}_2}{{\left\| {\Delta K} \right\|}_F}}}{{\sqrt 2  - 1 + \sqrt {1 - 2\left\| {{L^{ - 1}}} \right\|_2^2{{\left\| {\Delta K} \right\|}_F}} }}.
\end{eqnarray}
Note that $(\Delta L)^T = (\Delta L)^TL^{ - T} D_{m + n} (L^{ - T} D_{m + n})^{-1}$. Then, from (3.10), we have
\begin{eqnarray*}
{\left\| {\Delta L} \right\|_F} = {\left\| {(\Delta {L)^T}} \right\|_F} \le \frac{{\sqrt 2 {{\left\| {{L^{ - 1}}} \right\|}_2}{{\left\| {{L^{ - T}}{D_{m + n}}} \right\|}_2}{{\left\| {{{({L^{ - T}}{D_{m + n}})}^{ - 1}}} \right\|}_2}{{\left\| {\Delta K} \right\|}_F}}}{{\sqrt 2  - 1 + \sqrt {1 - 2\left\| {{L^{ - 1}}} \right\|_2^2{{\left\| {\Delta K} \right\|}_F}} }}.
\end{eqnarray*}
Since
\begin{eqnarray*}
&&\left\| {L^{ - T} D_{m + n} } \right\|_2 \left\| (L^{ - T} D_{m + n})^{-1}\right\|_2 =\left\| (D_{m + n}^{-1}L^{ T})^{-1} \right\|_2 \left\|  D_{m + n}^{-1}L^{ T}\right\|_2=\kappa_2( D_{m + n}^{-1}L^{ T})=\kappa_2(L D_{m + n}^{-1})
\end{eqnarray*}
and  $D_{m + n}  \in \mathbb{D}_{m + n}$ is arbitrary, we have the bound (3.3). The bound (3.4) follows from (3.3). \end{proof}

Now we give some remarks on this theorem, which are analogous to those in [7] on [7, Theorem 3.1].

\begin{remark}{\em
Taking the infimum of the expression below (21) in [2] over the set $\mathbb{D}_{m + n}$ and adding the higher-order term, we can derive the following first-order perturbation bound:
\begin{eqnarray}
\left\| {\Delta L} \right\|_F \le \left\| {L^{ - 1} } \right\|_2\left[ {\mathop {\inf }\limits_{{D_{m + n}}{ \in \mathbb{D}_{m + n}}} \kappa_2 (LD_{m + n}^{ - 1})} \right]{{{{\left\| {\Delta K} \right\|}_F}}} + O\left( {{{\left\| {\Delta K} \right\|_F^2}}} \right).
\end{eqnarray}
It is easy to find that the difference between the first-order bound (3.11) and the rigorous bound (3.4) is a factor of $2 + \sqrt 2$.}
\end{remark}

\begin{remark}{\em In [3, Theorem 2.3], the author obtained the following rigorous perturbation bound by the classic matrix equation approach:
\begin{eqnarray}
{{\left\| {\Delta L} \right\|_F }}\le \frac{{\sqrt 2 \left\| {L^{ - 1} } \right\|_2\kappa_2(L) \left\| {\Delta K} \right\|_F }}{{1 + \sqrt {1 - 2\left\| {L^{ - 1} } \right\|_F^2 \left\| {\Delta K} \right\|_F } }}
\end{eqnarray}
under the condition
\begin{equation*}
\left\| {L^{ - 1} } \right\|_F^2 \left\| {\Delta K} \right\|_F  < \frac{1}{2}.
\end{equation*}
The bound (3.12) is a little larger than
\begin{eqnarray}
{{\left\| {\Delta L} \right\|_F }} \le \frac{{\sqrt 2 \left\| {L^{ - 1} } \right\|_2\kappa_2(L) }}{{1 + \sqrt {1 - 2\left\| {L^{ - 1} } \right\|_2^2 \left\| {\Delta K} \right\|_F } }}{{\left\| {\Delta K} \right\|_F }}.
\end{eqnarray}
Setting $D_{m + n}  = I_{m + n}$ in (3.3) gives
\begin{eqnarray}
{{\left\| {\Delta L} \right\|_F }} \le \frac{{\sqrt 2 \left\| {L^{ - 1} } \right\|_2\kappa_2(L)  }}{{\sqrt 2  - 1 + \sqrt {1 - 2\left\| {L^{ - 1} } \right\|_2^2 \left\| {\Delta K} \right\|_F } }}{{\left\| {\Delta K} \right\|_F }}.
\end{eqnarray}
In comparison, we can find that the bound (3.14) is at most $\sqrt 2  + 1$ times as large as the bound (3.13). However, ${\mathop {\inf }\limits_{D_{m + n}  \in \mathbb{D}_{m + n} } \kappa_2\left( L{D_{m + n}^{ - 1} } \right)} $ can be arbitrarily smaller than
$\kappa_2(L)$ when $L$ has bad column scaling. For example, let $L=\left[{\begin{array}{*{20}c}
   1/\gamma & 0 \\
    1 & 1 \\
\end{array}} \right]$ with $0<\gamma  \ll1$. Then $\kappa_2\left( L{D_{m + n}^{ - 1} } \right) = \Theta(1)$ with $D_{m + n}  = {\rm diag}(1/\gamma,1)$, and $\kappa_2\left( L \right) = \Theta\left( 1/\gamma\right)$. Therefore, the bound (3.3) can be much sharper than (3.13) and hence (3.12).}
\end{remark}

\begin{remark}{\em The following rigorous perturbation bound given in [6, Theorem 3.1] was derived by the matrix-vector equation approach:
\begin{eqnarray}
{{\left\| {\Delta L} \right\|_F }} \le 2\left\| {\widehat{W}_{JL^T }^{ - 1} } \right\|_2 {{\left\| {\Delta K} \right\|_F }},
\end{eqnarray}
where ${\widehat{W}_{JL^T }} $ is a $\frac{{(m+n)(m+n + 1)}}{2} \times \frac{{(m+n)(m+n + 1)}}{2}$ lower triangular matrix defined by the elements of $JL^T $, under the condition
\begin{eqnarray*}
\left\| {\widehat{W}_{JL^T }^{ - 1} } \right\|_2 \left\| {\widehat{W}_{JL^T }^{ - 1} {\rm duvec}\left( {\Delta K} \right)} \right\|_2  < \frac{1}{4},
\end{eqnarray*}
which can be as bad as
\begin{eqnarray}
\left\| {\widehat{W}_{JL^T }^{ - 1} } \right\|_2^2 \left\| {\Delta K} \right\|_F  < \frac{1}{4}.
\end{eqnarray}
Please see [6, 9] for the specific structure of the matrix ${\widehat{W}_{JL^T }}$ and the definition of "duvec."

Numerical experiments indicated that the bound (3.15) is a little tighter than (3.4), however, it is not much tighter than (3.4). But the condition (3.16) can be much stronger than (3.1). For example, let {\small$L=\left[{\begin{array}{*{20}c}
   1 & 0 \\
   \gamma & 1 \\
\end{array}} \right]$} with $\gamma\gg 1$. Then $\left\| {\widehat{W}_{JL^T }^{ - 1} } \right\|_2^2 \approx \Theta\left( {\gamma ^4} \right)$ and $\left\| L^{ - 1} \right\|_2^2 = \Theta(\gamma^2)$. Moreover, it is more expensive to compute the bound (3.15) than that of (3.4).}
\end{remark}

\begin{remark}{\em The following rigorous perturbation bound was derived by the refined matrix equation approach [6, Theorem 3.2]:
\begin{eqnarray}
&&{{\left\| {\Delta L} \right\|_F }} \le \frac{{2\left\| L \right\|_2\kappa_2\left( L \right)\kappa_2\left( L{D_{m + n}^{ - 1}  } \right)\frac{{\left\| {\Delta K} \right\|_F }}{{\left\| K \right\|_2 }}}}{{1 + \sqrt {1 - 4\kappa_2\left( L \right)\left\| L \right\|_2 \left\| D_{m + n}  L^{ - 1} \right\|_2 \left\| {D_{m + n}^{ - 1} } \right\|_2 \frac{{\left\| {\Delta K} \right\|_F }}{{\left\| K \right\|_2 }}} }}\nonumber\\
&&\quad\quad\quad\ \ \leq2\left\| {L } \right\|_2\kappa_2\left( L \right)\kappa_2\left( L{D_{m + n}^{ - 1}  } \right)\frac{{\left\| {\Delta K} \right\|_F }}{{\left\| K \right\|_2 }}
\end{eqnarray}
under the condition
\begin{eqnarray}
\kappa_2\left( L \right)\left\| L \right\|_2 \left\| { D_{m + n} L^{ - 1} } \right\|_2 \left\| {D_{m + n}^{ - 1} } \right\|_2 \frac{{\left\| {\Delta K} \right\|_F }}{{\left\| K \right\|_2 }} < \frac{1}{4}.
\end{eqnarray}
Minimizing $\kappa_2\left( L{D_{m + n}^{ - 1}  } \right)$ over the set $\mathbb{D}_{m + n}$, we can make the bound (3.17) similar to the new bound (3.4). However, noting the facts $\left\| K \right\|_2\leq\left\| L \right\|_2^2$ and $\left\| { D_{m + n} L^{ - 1} } \right\|_2\geq\left\|  L^{ - 1}  \right\|_2/\left\| {D_{m + n}^{ - 1} } \right\|_2$, we have
\begin{eqnarray*}
\kappa_2\left( L \right)\left\| L \right\|_2 \left\| { D_{m + n} L^{ - 1} } \right\|_2 \left\| {D_{m + n}^{ - 1} } \right\|_2 \frac{{\left\| {\Delta K} \right\|_F }}{{\left\| K \right\|_2 }}  \ge \left\| L^{-1} \right\|_2^2\left\| {\Delta K} \right\|_F.
\end{eqnarray*}
Therefore, the condition (3.18) is not only complicated but also more constraining than (3.1). Especially, when we minimize the bound (3.17) over the set $\mathbb{D}_{m + n}$, the best choice of $D\in \mathbb{D}_{m+n}$ may make the condition (3.18) worse.}
\end{remark}

From the above discussions, we can find that the bounds given in Theorem 3.1 have more advantages compared with the existing ones.

\section{Rigorous perturbation bounds with componentwise perturbation}

According to Algorithm 1 for computing the generalized Cholesky factorization given in [1], [1, Eqns. (8)--(10)], and [10, Lemma 8.4 and Theorem 10.3],  we have that if
\begin{eqnarray*}
K + \Delta K = \widetilde L{J_{m + n}}{\widetilde L^T},\quad \widetilde{L}=\left(
                                                                             \begin{array}{cc}
                                                                               \widetilde{L}_{11} & 0 \\
                                                                               \widetilde{L}_{21}&  \widetilde{L}_{22} \\
                                                                             \end{array}
                                                                           \right),
\quad \Delta K=\left(
                                                                   \begin{array}{cc}
                                                                     \Delta A & (\Delta B)^T \\
                                                                     \Delta B & \Delta C \\
                                                                   \end{array}
                                                                 \right),
\end{eqnarray*}
then
\begin{eqnarray*}
&&|\Delta A| \leq \gamma_{3m+1}|\widetilde{L}_{11}||\widetilde{L}_{11}^T|,\
|\Delta B| \leq \gamma_{3m+1}|\widetilde{L}_{21}||\widetilde{L}_{11}^T|,\
|\Delta C| \leq\gamma_{3n+1}|\widetilde{L}_{22}||\widetilde{L}_{22}^T|,
\end{eqnarray*}
where $\gamma_i=iu/(1-iu)$ and $u$ is the unit roundoff.
Thus, the computed generalized Cholesky factor $\widetilde{L}$ satisfies
\begin{eqnarray}
K + \Delta K = \widetilde L{J_{m + n}}{\widetilde L^T},\quad |{\Delta K} | \le \varepsilon | {\widetilde L^T} || \widetilde {L} |,
\end{eqnarray}
where $\varepsilon={\min}\{\gamma_{3m+1}, \gamma_{3n+1}\}$.

In the following, we consider the rigorous perturbation bounds for the generalized Cholesky factorization with the perturbation $\Delta K$ having the same form as in (4.1).

\begin{theorem}
Let $\Delta K \in {\mathbb{R}^{\left( {m + n} \right) \times \left( {m + n} \right)}}$ be a symmetric perturbation in $K \in {\mathbb{R}^{\left( {m + n} \right) \times \left( {m + n} \right)}}$ which has the same form as that in {\rm (1.1)}, however, does not necessarily have the generalized Cholesky factorization, and $K + \Delta K$ have the generalized Cholesky factorization {\rm(4.1)}. If
\begin{eqnarray}
{\mathop{\rm cond}_F\nolimits} ( {\widetilde L} ){\mathop{\rm cond}_F\nolimits} ( {{{\widetilde L}^{ - T}}} )\varepsilon  < \frac{1}{2},
\end{eqnarray}
then $K$ has the generalized Cholesky factorization $K = LJ_{m + n} L^T$, where $L=\widetilde L-\Delta L$. Moreover,
\begin{eqnarray}
&&{{{{\left\| {\Delta L} \right\|}_2}}}\le \frac{{\sqrt 2{{\left(\mathop {\inf }\limits_{{D_{m + n}} \in {\mathbb{D}_{m + n}}} {{\left\| {\widetilde LD_{m + n}^{ - 1}} \right\|}_2}{{\left\| {{D_{m + n}}|{{\widetilde L}^{ - 1}}|| {\widetilde L} |} \right\|}_2}\right){\mathop{\rm{cond}}_F\nolimits} ( {\widetilde L} )\varepsilon }}}}{{\sqrt 2  - 1 + \sqrt {1 - 2{\rm{cond}}_F( {\widetilde L} ){\rm{cond}}_F( {{{\widetilde L}^{ - T}}} )\varepsilon } }}\\
&&\quad\quad\quad\ \le {{( {2 + \sqrt 2 } )\left(\mathop {\inf }\limits_{{D_{m + n}} \in {\mathbb{D}_{m + n}}} {{\left\| {\widetilde LD_{m + n}^{ - 1}} \right\|}_2}{{\left\| {{D_{m + n}}|{{\widetilde L}^{ - 1}}|| {\widetilde L} |} \right\|}_2}\right){\mathop{\rm{cond}}_F\nolimits} ( {\widetilde L} \textbf{})\varepsilon }}.
\end{eqnarray}
\end{theorem}
\begin{proof}
From (2.1) and (4.1),  for any $t \in \left[ {0,1} \right]$, we have
\begin{eqnarray*}
&&\left\| {{\widetilde L^{ - 1}}t(\Delta K){\widetilde L^{ - T}}} \right\|_F \le \left\|{| {\widetilde L^{ - 1}}|| {\widetilde L} || {{{\widetilde L^T}}} ||{\widetilde L^{ -T}}|} \right\|_F\varepsilon
\le {\mathop{\rm{cond}}_F\nolimits}( {\widetilde L} ){\mathop{\rm{cond}}_F\nolimits}( {{{\widetilde L}^{-T}}})\varepsilon < \frac{1}{2}.
\end{eqnarray*}
Thus, considering the proof of Theorem 3.1, for any $t \in \left[ {0,1} \right]$, $\left( {K + \Delta K} \right) - t(\Delta K) $ has the following generalized Cholesky factorization
\begin{eqnarray}
\left( {K + \Delta K} \right) - t(\Delta K)=\widetilde L (t) J_{m + n}\widetilde L^T (t)= ( {\widetilde L - (\Delta L( t ))} ){J_{m + n}}{( {\widetilde L - ( \Delta L(t) )} )^T},
\end{eqnarray}
where $\widetilde L (t)$ is lower triangular having the same structure as that of $L$ in (1.2) and $\Delta L( t )=\widetilde L-\widetilde L(t)$ with $\Delta L( 0)=0$. Setting $\Delta L\left( 1 \right) = \Delta L$ in (4.5) leads to the generalized Cholesky factorization of $K$ as in (1.2).

Next, we prove (4.3) and (4.4). The proof is similar to the one for the bounds (3.3) and (3.4). Considering (4.1), from (4.5), we have
\begin{eqnarray}
{J_{m + n}}(\Delta {L}( t ))^T{\widetilde L^{ - T}} = {\mathop{\rm up}\nolimits} \left[{t{{\widetilde L}^{ - 1}}(\Delta K){{\widetilde L}^{ - T}} + {{\widetilde L}^{ - 1}}(\Delta L( t )){J_{m + n}}(\Delta {L}( t ))^T{{\widetilde L}^{ - T}}} \right].
\end{eqnarray}
Taking the Frobenius norm on (4.6), and using (2.3), (2.1) and the bound of $\left| {\Delta K} \right|$ in (4.1) gives
\begin{eqnarray*}
&&{\left\| {{{\widetilde L}^{ - 1}}(\Delta L( t ))} \right\|_F}={\left\|{J_{m + n}} {(\Delta {L}( t ))^T{{\widetilde L}^{ - T}}} \right\|_F} \le \frac{1}{{\sqrt 2 }}\left[ {{\rm{cond}}_F(\widetilde L){\rm{cond}}_F({{\widetilde L}^{ - T}})\varepsilon t + \left\| {{{\widetilde L}^{ - 1}}(\Delta L(t))} \right\|_F^2} \right].
\end{eqnarray*}
That is,
\begin{eqnarray*}
 x^2(t)-\sqrt{2}x(t)+c(t) \ge 0,
\end{eqnarray*}
where $x(t)=\left\| {{{\widetilde L}^{ - 1}}(\Delta L( t ))} \right\|_F$ and $c(t)={\mathop{\rm{cond}}_F\nolimits}( {\widetilde L} ){\mathop{\rm{cond}}_F\nolimits}({\widetilde L^{ - T}}{\rm{ )}}\varepsilon t$.
Considering (4.2), we can check that the conditions of Lemma 2.1 hold.
Thus,
\begin{eqnarray}
{\left\| {{{\widetilde L}^{ - 1}}(\Delta L)} \right\|_F} \le \frac{1}{{\sqrt 2 }}\left({1 - \sqrt {1 - 2{\mathop{\rm{cond}}_F\nolimits}( {\widetilde L} ){\mathop{\rm{cond}}_F\nolimits}( {{{\widetilde L}^{-T}}} )\varepsilon } }\right).
\end{eqnarray}

Taking $t=1$ in (4.6) and considering (2.4), for any ${D_{m + n}} \in {{\mathbb{D}}_{m + n}}$, we have
\begin{align}
{J_{m + n}}(\Delta {L)^T}{\widetilde L^{ - T}}{D_{m + n}} = {\rm{up}}\left( {{{\widetilde L}^{ - 1}}(\Delta K){{\widetilde L}^{ - T}}{D_{m + n}} + {{\widetilde L}^{ - 1}}(\Delta L){J_{m + n}}(\Delta {L)^T}{{\widetilde L}^{ - T}}{D_{m + n}}} \right).
\end{align}
Taking the Frobenius norm on (4.8) and using (2.2) and (2.1) yields
\begin{eqnarray*}
&&{\left\| {{D_{m + n}}{{\widetilde L}^{ - 1}}(\Delta L)} \right\|_F} \le {\left\| {{{\widetilde L}^{ - 1}}(\Delta K){{\widetilde L}^{ - T}}}{D_{m + n}} \right\|_F} + {\left\| {{{\widetilde L}^{ - 1}}(\Delta L)} \right\|_F}{\left\| {{D_{m + n}}{{\widetilde L}^{ - 1}}(\Delta L)} \right\|_F}.
\end{eqnarray*}
Further, substituting the bound of $\left| {\Delta K} \right|$ in (4.1) into the above inequality leads to
\begin{eqnarray*}
&&{\left\| {{D_{m + n}}{{\widetilde L}^{ - 1}}(\Delta L)} \right\|_F} \le {\left\| {{{\widetilde L}^{ - 1}}
| {\widetilde L} || {{{\widetilde L}^T}} |{{\widetilde L}^{ - T}}{D_{m + n}}} \right\|_F} \varepsilon + {\left\| {{{\widetilde L}^{ - 1}}(\Delta L)} \right\|_F}{\left\| {{D_{m + n}}{{\widetilde L}^{ - 1}}(\Delta L)} \right\|_F}\\
&&\quad\quad\quad\quad\quad\quad\quad\le {\left\| {{D_{m + n}}|{{\widetilde L}^{ - 1}}|| {\widetilde L} |} \right\|_2} {\mathop{\rm{cond}}_F\nolimits}( {{{\widetilde L}}} )\varepsilon + {\left\| {{{\widetilde L}^{ - 1}}(\Delta L)} \right\|_F}{\left\| {{D_{m + n}}{{\widetilde L}^{ - 1}}(\Delta L)} \right\|_F}.
\end{eqnarray*}
Noting (4.7), we obtain
\begin{eqnarray*}
{\left\| {{D_{m + n}}{{\widetilde L}^{ - 1}}(\Delta L)} \right\|_F} \le \frac{{\sqrt 2 {{\left\| {{D_{m + n}}|{{\widetilde L}^{ - 1}}|| {\widetilde L} |} \right\|}_2}{\mathop{\rm{cond}}_F\nolimits}( {{{\widetilde L}}} )\epsilon}}{{\sqrt 2  - 1 + \sqrt {1 - 2{\mathop{\rm{cond}}_F\nolimits}( {\widetilde L} ){\mathop{\rm{cond}}_F\nolimits}( {{{\widetilde L}^{-T}}} )\varepsilon } }},
\end{eqnarray*}
which combined with the fact
\begin{eqnarray*}
{\left\| {\Delta L} \right\|_F} \le {\left\| {\widetilde LD_{m + n}^{ - 1}} \right\|_2}{\left\| {{D_{m + n}}{{\widetilde L}^{ - 1}}(\Delta L)} \right\|_F}
\end{eqnarray*}
implies the bound (4.3) and then (4.4).
\end{proof}

\begin{remark}{\em The following first-order perturbation bound can be derived from (4.3),
\begin{eqnarray}
&&{{{{\left\| {\Delta L} \right\|}_2}}}\le {{\left(\mathop {\inf }\limits_{{D_{m + n}} \in {\mathbb{D}_{m + n}}} {{\left\| {\widetilde LD_{m + n}^{ - 1}} \right\|}_2}{{\left\| {{D_{m + n}}|{{\widetilde L}^{ - 1}}|| {\widetilde L} |} \right\|}_2}\right){\mathop{\rm{cond}}_F\nolimits} ( {\widetilde L} \textbf{})\varepsilon }}+O(\varepsilon^2).
\end{eqnarray}
In addition, it is worthy pointing out that the perturbation bounds obtained in this section are unlike the ones in Section 3. These bounds involve the generalized Cholesky factor of $K+\Delta K$ but not the one of $K$. This is because the bound of $|\Delta K|$ in (4.1) involves the generalized Cholesky factor of $K+\Delta K$.}
\end{remark}
\section{Concluding remarks}

In this paper, some new rigorous perturbation bounds for the generalized Cholesky factorization with normwise or componentwise perturbations in the given matrix are obtained. These bounds not only have simple and moderate conditions but also can be much smaller than some existing ones. To estimate these bounds efficiently, the suitable scaling matrix $D$ is needed. In [9], the author provided some methods, which are also applicable to the bounds in this paper. Please refer to [9] for detail on these methods.

\end{document}